\documentclass[12pt, leqno]{amsart}
\usepackage{amssymb}
\usepackage[dvips]{color}
\usepackage{amsmath}
\usepackage{amsthm}
\usepackage[a4paper, margin=4cm]{geometry}
\usepackage{amsfonts}
\usepackage[all]{xypic} 
\usepackage{bbm}
\usepackage{xcolor}

\usepackage{hyperref}
\usepackage{url}

\allowdisplaybreaks[4] 
\newtheoremstyle{myremark}     {10pt}{10pt}{}{}{\bfseries}{.}{.5em}{}

\newtheorem{thm}{Theorem}[section]

\newtheorem{lem}[thm]{Lemma}
\newtheorem{prop}[thm]{Proposition}
\theoremstyle{definition}
\newtheorem{defn}[thm]{Definition}

\theoremstyle{myremark}

\numberwithin{equation}{section}

\newcommand{\h}{\mathcal{H}}

\newcommand{\CP}{\mathcal{P}}
\newcommand{\CB}{\mathcal{B}}
\newcommand{\CQ}{\mathcal{Q}}
\newcommand{\CI}{\mathcal{I}}

\newcommand{\CD}{\mathcal{D}}

\newcommand{\CE}{\mathcal{E}}

\newcommand{\CM}{\mathcal{M}}
\newcommand{\CN}{\mathcal{N}}

\newcommand{\C}{\mathbb{C}}
\newcommand{\R}{\mathbb{R}}

\newcommand{\E}{\mathbb{E}}
\newcommand{\Z}{\mathbb{Z}}
\newcommand{\N}{\mathbb{N}}

\newcommand{\abs}[1]{\left\vert#1\right\vert}

\newcommand{\norm}[1]{\left\Vert#1\right\Vert}

\newcommand{\esssup}{\mathrm{esssup}}

\begin{document}
	
	\title[non-commutative square function]{Weighted weak $(1,1)$ estimate for non-commutative square function}
	
	\author[S. Ray]{Samya Kumar Ray}
	\address{ School of Mathematics, Indian Institute of Science Education and Research Thiruvananthapuram, Kerala - 695551}
	\email{samya@iisertvm.ac.in, samyaray7777@gmail.com}
	
	\author[D. Saha]{Diptesh Saha}
	\address{Statistics-Mathematics Unit, Indian Statistical Institute, Kolkata, India}
	\email{dptshs@gmail.com}

	
	

	\begin{abstract}
		In this article, we consider weighted weak type $(1,1)$ inequality for certain square function associated to differences of ball averages and martingale in the non-commutative setting. This establishes a weighted version of main result of \cite{hong2021noncommutative}.
	\end{abstract}
	
	\maketitle 
	\section{Introduction}
	A fundamental problem in harmonic analysis and measure theoretic ergodic theory is to study  weak and strong type inequalities of various operators. Though it is often motivated by the question of pointwise convergence of different averages, it has also other wide range of applications. A rather quantitative approach for these kinds of problems is to establish the so called `square function' inequalities.
	In this article, we study weighted versions of certain square function inequalities in the non-commutative setting motivated from ergodic theory.

	
	Let us consider the measure space $(\R^d, \CB, \mu)$ where $\CB$ is the Borel $\sigma$-algebra and $\mu$ the translation invariant Borel measure on $\mathbb{R}^d.$ Let $n\in\mathbb{Z}.$ Fix the $n$-th dyadic filtration $\CB_n$ generated by the dyadic cubes of side-length $2^{-n}$.  Hence for each $n\in\mathbb{Z},$ we have the associated classical conditional expectation $\E_n$. 
	Let $\mathcal M$ be a semifinite von Neumann algebra with a f.n.s. trace $\tau$. Then $\E_n\otimes I_{\mathcal M}$ is a conditional expectation on $\mathcal N:=L^\infty(\mathbb R^d)\overline{\otimes} \mathcal M$ which we denote by $\CE_n.$ Note that $L^\infty(\mathbb R^d)\overline{\otimes} \mathcal M$ is again a semifinite von Neumann algebra with f.n.s. trace $\varphi:=\int_{\mathbb{R^d}}\otimes \tau.$ Consider the non-commutative $L^1$-space $L^1(\CM, \tau)$. Now for any locally integrable function $f: \R^d \to L^1(\CM, \tau)$, we consider the averaging operator
	\begin{equation}\label{avg op}
		M_t f (x)= \frac{1}{\mu(B_t)} \int_{B_t} f(x+y) d\mu(y), ~ x \in \R^d,
	\end{equation}
	where, for all $t \in \R$, $B_t$  is the ball of radius $2^{-t}$ with center at the origin. The non-commutative square function that we are going to consider is associated to the following sequence of operators,
	\begin{equation}\label{nc1.2}
		T_n f(x):= (M_n - \CE_n) f(x), ~ x \in \mathbb{R}^d.
	\end{equation}
	If we replace the von Neumann algebra $\CM$ by the set of complex numbers $\C$, then in this case the square function is formulated as
	\begin{equation}
		Lf(x):= \Big( \sum_{k} |(M_k - \mathbb E_k) f(x)|^2 \Big)^{\frac{1}{2}}, ~ x \in \mathbb{R}^d.
	\end{equation}
	

	To obtain a similar result, we need an appropriate norm in the non-commutative setup. For this purpose, one has to consider the row and column spaces. The row space (resp. column space) is defined as the closure of the finite sequences in $L^p(\mathcal{N})$ with respect to the norm, defined as
	\begin{equation}
		\left\Vert (f_k) \right\Vert_{L^p(\mathcal N;\ell^r_2)}:= \left\Vert \Big(\sum_k |f_k^*|^2\Big)^{1/2} \right\Vert_p ~\left( \text{resp.} \left\Vert (f_k) \right\Vert_{L^p(\mathcal N;\ell^c_2)}:= \left\Vert \Big(\sum_k |f_k|^2\Big)^{1/2} \right\Vert_p \right).
	\end{equation}
	
	The associated row-column space  is defined as the the following.
	\begin{equation}
		L^p(\mathcal N;\ell^{rc}_2):= 
		\begin{cases}
			L^p(\mathcal N;\ell^{r}_2) \cap L^p(\mathcal N;\ell^{c}_2),~ if~ 2 \leq p \leq \infty \\
			\text{with} \left\Vert \cdot \right\Vert:= \max \{ \left\Vert \cdot \right\Vert_{L^p(\mathcal N;\ell^{r}_2)}, \left\Vert \cdot \right\Vert_{L^p(\mathcal N;\ell^{c}_2)} \}\\\\
			L^p(\mathcal N;\ell^{r}_2)+ L^p(\mathcal N;\ell^{c}_2),~ if~ 1\leq p<2\\ \text{with } \left\Vert (f_k) \right\Vert:= \inf\limits_{f_k=g_k+ h_k} \{ \left\Vert (g_k) \right\Vert_{L^p(\mathcal N;\ell^{c}_2)} + \left\Vert (h_k) \right\Vert_{L^p(\mathcal N;\ell^{r}_2)} \}
		\end{cases}
	\end{equation}
	
	
	Similarly one can also define weak row, column, row-column spaces with natural definitions as above. In \cite{hong2021noncommutative}, the authors proved the strong and weak type inequalities for non-commutative square function associated to the sequence of operators $(T_k)_{k\in\mathbb{Z}}$.

	One of the most active areas of research in classical analysis is to obtain weighted analogues of various results in classical analysis. Hence, it is natural to ask for a weighted analogue of the result above. In the classical setting, this direction of research was started with the celebrated work of \cite{muckenhoupt1972weighted} and continued further by \cite{fefferfeffer12} and many others. Moreover, the question about whether there is a linear dependence of the norm inequalities on the $A_2$-constant became the famous $A_2$-conjecture, which was finally settled by Hyt\"{o}nen \cite{Hytonen12ann}. It is due to development of operator space theory and non-commutative probability theory, non-commutative analysis have been developing in a great speed in the past few decades and many results in non-commutative harmonic analysis and ergodic theory were established in \cite{Junge2007}, \cite{Mei07}, \cite{Homglaixu23}, \cite{Hongraywang23}, \cite{Hongliaowang21}, \cite{hong2022noncommutative} etc. Despite these remarkable development, there are not many results available in the weighted setting. Recently, in \cite{galkazka2022sharp}, the authors have studied non-commutative Doob's inequality in the weighted setting and also studied some singular operators associated to nice kernels. In this paper, we prove the following result.
	\begin{thm}\label{thm1.1}
		Let $1 \leq p < \infty$ and $w$ be an $A_1$-weight. Then there exists a constant $C_p(w)$, depending only on $p$ and $w$ such that
		\begin{equation}
			\begin{cases}
				\left\Vert (T_k f) \right\Vert_{L^{1, \infty}(\mathcal N_{w};\ell^{rc}_2)} \leq C_1(w) \left\Vert f \right\Vert_{1,w},~ \forall f \in L^1(\mathcal N_w)~ \text{and,}\\\\
				\left\Vert (T_k f) \right\Vert_{L^p(\mathcal N_w;\ell^{rc}_2)} \leq C_p(w) \left\Vert f \right\Vert_{p,w}~ \forall f \in L^p(\mathcal N_w),~ \text{when } 1<p< \infty.
			\end{cases}
		\end{equation}   
	\end{thm}
	We refer section \eqref{prelim} for any unexplained notation.	Note that by \cite{krause2018weighted}, the above theorem is easy to prove  for $p=2$ as $L_2(\mathcal M)$ is a Hilbert space. For other values of $p$ we linearize the problem by introducing the operator $T$ as in  eq. \ref{linearization}. Then by non-commutative Khintchine inequalities, to prove weak type $(1,1)$ of the square function, it is enough to prove weak type $(1,1)$ inequality for $T$. The strong $(p,p)$ then follows from a routine argument (see Theorem \eqref{strongbdd}). To prove weak type $(1,1)$ for $T$, our method relies on the Calder\' on-Zygmund decomposition as in \cite{hong2021noncommutative}. However, unlike \cite{hong2021noncommutative}, we use recent Calder\' on-Zygmund decomposition developed by \cite{cadilhac2022spectral} and \cite{cadilhac2022noncommutative}. One more ingredient is weighted Doob's inequality which was studied in \cite{galkazka2022sharp}. We have tried our best to keep track of the constants and dependence on $w$ as explicit as possible in the first inequality in Theorem \eqref{thm1.1}. However, it seems to be a challenging problem to obtain the best dependence on $w$ even in the classical situation. Although the constants denoted by $C(w),C_1(w)$ etc, and similar notation appear throughout the paper, they are not necessarily identical and may vary from one occurrence to another, depending on the context.
	
	Let us briefly discuss how the rest of the article has been
	organized. In Section \ref{prelim} we discuss all the basic definitions and concepts required. This includes the notion of weights and non-commutative $L_p$-spaces. In Section \ref{weaktypeesti}, we recall the Calder\'on-Zygmund decomposition and the proof of the weighted inequality of the square function. After submitting our manuscript to the arxiv, Dejian Zhou sent us a file where he with J. Cao and D. Zhou et al. also independently obtained Theorem 1.1.
\section{Preliminaries}\label{prelim}
	Throughout this article, $\CM (\subseteq \CB(\h))$ is assumed to be a separable von Neumann algebra equipped with a faithful, normal, semifinite (f.n.s.) trace $\tau$. The lattice of projections on the von Neumann algebra $\CM$ is denoted by $\CP(\CM)$. By $\CM'$, we will denote the commutant of the von Neumann algebra $\CM$, which is also a separable von Neumann algebra defined on the Hilbert space $\h$.
	
	\subsection{Non-commutative $L^p$-spaces}
	Let $\CM$ be a semi-finite von Neumann algebra with a f.n.s trace $\tau$. An operator (possibly unbounded) $x$, which is closed and densely defined on $\h$, is said to be affiliated to $\CM$ if it commutes with all unitaries $u'$ of $\CM'$. The same operator $x$ is called $\tau$-measurable if for all $\delta>0$ there exists  $p \in \CP(\CM)$ such that $p\h \subseteq \CD(x)$ (domain of $x$) and $\tau(1-p)< \delta$. The space of all $\tau$-measurable operators, which are affiliated to $\CM$ is denoted by $L^0(\CM, \tau)$. It is well-known that  $\tau$ can be extended to the positive cone $L^0(\CM, \tau)_+$ and one can define for $x \in L^0(\CM, \tau)$
	\begin{align*}
		\norm{x}_p:= (\tau(\abs{x}^p))^{1/p}, ~ 0 < p< \infty,
	\end{align*}
	where $\abs{x}:= (x^*x)^{\frac{1}{2}}$.
	Then the non-commutative $L^p$ space associated to $(\CM, \tau)$ for $0<p< \infty$ is defined as 
	\begin{align*}
		L^p(\CM, \tau):= \{x \in L^0(\CM, \tau): \norm{x}_p< \infty\},
	\end{align*}
	and  $L^\infty(\CM):= \CM$. Furthermore, it is a Banach space with respect to the norm $\norm{\cdot}_p$ when $1\leq p\leq\infty$.
	
	On the other hand, the non-commutative weak $L^p$-space associated to $(\CM, \tau)$ for $0<p< \infty$ is defined as a subspace of $L^0(\CM, \tau)$ such that for $x \in L^0(\CM, \tau)$ the following quasi-norm
	\begin{align*}
		\norm{x}_{p, \infty}:= \sup_{\lambda>0} \lambda \tau(\chi_{(\lambda, \infty)}(\abs{x}))^{\frac{1}{p}}
	\end{align*}
	is finite. We denote the collection as $L^{p, \infty}(\CM)$. We note that for $x_1, x_2 \in L^{1, \infty}(\CM)$ and $\lambda>0$,
	\begin{equation}
		\tau(\chi_{(\lambda, \infty)}(\abs{x_1+x_2})) \leq \tau(\chi_{(\lambda/2, \infty)}(\abs{x_1})) + \tau(\chi_{(\lambda/2, \infty)}(\abs{x_2})).
	\end{equation}
	
	\subsection{Non-commutative Martingales}
	Let $(\CM, \tau)$ be a semifinite von Neumann algebra as described before. Let $d \in \N$ be fixed and we consider the von Neumann algebra $\CN:= L^\infty(\R^d, \CB, \mu) \otimes \CM$, where $\R^d$ is equipped with the Lebesgue measure $\mu$ and Borel $\sigma$-algebra $\CB$. Clearly, $\CN$ is equipped with the f.n.s. trace $\varphi := (\int \cdot~ d\mu) \otimes \tau$. For $k, n_1, \ldots, n_d \in \Z$ consider the dyadic cube in $\R^d$;
	\begin{align*}
		[2^kn_1, 2^k(n_1+1)) \times \cdots \times [2^kn_d, 2^k(n_d+1)).
	\end{align*}
	By $\CQ$, we denote the collection of all dyadic cubes in $\R^d$. For any $k \in \Z$, $\CQ_k$ will denote the collection of dyadic cubes of side length $2^{-k}$. Therefore, $\abs{Q}$, the volume of a $Q \in \CQ_k$ is $2^{-dk}$. Let $k \in \Z$ and  $\CB_k$ be the $\sigma$-algebra generated by the dyadic partition $\CQ_k$. Now, for every $k \in \Z$ we define the von Neumann subalgebra
	\begin{align*}
		\CN_k:= L^\infty(\R^d, \CB_k, \mu) \otimes \CM
	\end{align*}
	of $\CN$. Now since $\varphi|_{\CN_k}$ is again semi-finite, there exists a normal conditional expectation $\CE_k$ which is of the form $\CE_k:= \E_k(\cdot | \CB_k) \otimes I_{\CM}$, where $\E_k(\cdot | \CB_k)$ is the conditional expectation from $L^\infty(\R^d, \CB, \mu)$ onto $L^\infty(\R^d, \CB_k, \mu)$. Therefore, for $k \in \Z$ and $f \in L^1(\CN, \varphi)$ we have
	\begin{align*}
		\CE_k(f)= \sum_{Q \in \CQ_k} \Big( \frac{1}{\abs{Q}} \int_Q f \Big) \chi_Q.
	\end{align*}
	Moreover, observe that $(\CN_k)_{k \in \Z}$ is an increasing family of subalgebras of $\CN$. Hence, for every $f \in L^1(\CN, \varphi)$, the sequence $(\CE_k(f))_{k \in \Z}$ forms a non-commutative martingale, that is 
	\begin{align*}
		\CE_j(\CE_k(f))= \CE_j(f), ~ \text{for } j \leq k.
	\end{align*}
	
	\subsection{Martingale weights and weighted $L^p$ spaces}
	
	Consider the measure space $(\R^d, \CB, \mu)$ as described above. A weight $w$ is a positive, integrable function on $(\R^d, \CB, \mu)$. For any $C \in \CB$, we will write $w(C)= \int_C w d\mu$. For $1<p< \infty$, the weight $w$ is said to satisfy the Muckenhoupt's $A_p$ condition if 
	\begin{align*}
		[w]_{A_p}:=\sup_{Q} \frac{\int_Q w d \mu}{\mu(Q)} \Big(\frac{\int_Q w^{1/(1-p)} d \mu}{\mu(Q)} \Big)^{p-1} < \infty,
	\end{align*}
	and $w$ is said to satisfy the Muckenhoupt's $A_1$ condition if 
	\begin{equation*}
		[w]_{A_1}:=\sup_{Q} \esssup_{x \in \R^d} \frac{\int_Q w d \mu / \mu(Q)}{w(x)}< \infty,
	\end{equation*}
	where in both cases the supremum is taken over all cubes $Q \subset \R^d$ with sides parallel to axes. For a measurable function $f$ on $(\R^d, \CB, \mu)$ define the Hardy-Littlewood maximal function as
	\begin{align*}
		Mf(x):= \sup_{Q} \frac{1}{\mu(Q)} \int_Q \abs{f(y)} d\mu(y); ~ x \in \R^d,
	\end{align*}
	(where the supremum is taken over all cubes $Q \subset \R^d$ with sides parallel to axes containing $x$)
	and recall the following theorem (see \cite{duoandikoetxea2024fourier}).
	\begin{thm}\label{weighted max ineq classical}
		For $1\leq p< \infty$, the weak type $(p,p)$ inequality
		\begin{align*}
			w \Big( \{x \in \R^d: Mf(x)> \lambda \} \Big) \leq \frac{C}{\lambda^p} \int_{\R^d} \abs{f(y)}^p w(y) d \mu(y)
		\end{align*}
		is true if and only if $w$ satisfies the Muckenhoupt's $A_p$ condition.
	\end{thm}
	
	\noindent Then, it follows from Theorem \ref{weighted max ineq classical} that for any $S, Q \in \CB$ with $S \subset Q$ and $\int_Q f d\mu >0$
	\begin{equation}\label{weight prop}
		\frac{w(Q)}{\mu(Q)} \leq [w]_{A_1} \frac{w(S)}{\mu(S)}
	\end{equation}
	and 
	\begin{equation*}
		\frac{w(Q)}{\mu(Q)} \leq [w]_{A_1} \inf_{x \in Q} w(x).
	\end{equation*}
	
	Furthermore, if $w$ satisfies Muckenhoupt's $A_1$ condition,  then there exists $0< \delta< 1$ (depending on $[w]_{A_1}$) and a constant $C(w)>0$ such that for all $S \subseteq Q$
	\begin{equation*}
		\frac{w(S)}{w(Q)} \leq C(w) \left( \frac{\mu(S)}{\mu(Q)} \right)^\delta.
	\end{equation*}
	
The above properties can be found in \cite{duoandikoetxea2024fourier}. In this article, we will consider weights of the form $w \otimes I_{\CM}$, where $w$ is a classical weight as defined above. Observe that $w \otimes I_{\CM}$ commutes with all elements of $\CN$ and we say $w \otimes I_{\CM}$ satisfy Muckenhoupt's $A_1$ condition if $w$ satisfies the same. Furthermore, we define $[w \otimes I_{\CM}]_{A_1}:= [w]_{A_1}$. Therefore, from now on we will only use the notation $w$ for such weights without any ambiguity. For such an weight $w$, let us recall that the weighted non-commutative $L^p$-space ($1 \leq p < \infty$) is defined as 
	\begin{align*}
		L^p_w(\CN):= \{x \in L^0(\CN, \varphi): xw^{1/p} \in L^p(\CN, \varphi)\}
	\end{align*}
	In other words, if we define the weighted f.n.s. trace $\varphi_w(\cdot):= \varphi(\cdot w)$ on the von Neumann algebra $\CN$, then it is eminent that $L^p_w(\CN)$ is the non-commutative $L^p$-space on $(\CN, \varphi_w)$. In the sequel we will sometime write $\CN_w$ to emphasise that the von Neumann algebra $\CN$ is considered along with the trace $\varphi_w$. The associated Banach space norm in $L^p_w(\CN)$ will be denoted by $\norm{\cdot}_{p,w}$.
	
	\subsection{Non-commutative Calder\'on-Zygmund decomposition}
	
	Let $f \in L^1(\CN, \varphi_w) \cap \CN$ be positive and $x\mapsto\|f(x)\|_1$ is compactly supported. More precisely, observe that the set 
	\begin{align*}
		\CN_{c, +}:= L^1(\CN, \varphi_w) \cap \{ f:\R^d \to \CM: f \in \CN_+,~ \overrightarrow{\text{supp}} f \text{ is compact } \}
	\end{align*}
	is a dense subset of $L^1(\CN, \varphi_w)_+$, where $\overrightarrow{\text{supp}} f$ is defined as the support of the map $x\mapsto\|f(x)\|_1$. From now on, we confine ourselves to the above dense set to obtain the desired estimates. Furthermore, it is well-known that for all $f \in \CN_{c, +}$ and $\lambda>0$, there exists $m_\lambda(f) \in \Z$ such that $\CE_k(f) \leq \lambda 1_\CN$ for all $k \leq m_\lambda(f)$. In the following, we will fix $f \in \CN_{c, +}$ and $\lambda >0$ and assume without loss of generality $m_\lambda(f)=0$.
	
	Consider the von Neumann algebra  $\CN$ and the weight $w$ satisfying Mukenhopt's  $A_1$ condition. We further recall the dyadic filtration $\{ \CN_n \}_{n \geq 0}$ as described above and consider the martingale $\{\CE_n(f)\}_{n \geq 0}$. Then by Cuculescu's construction (\cite{cuculescu1971martingales}),  there exists a decreasing sequence of projections $(q_n)_{n \in \Z}$ in $\CN$, with $q_n=1$ for all $n \leq 0$ and for $n >0$
	\begin{align*}
		q_n:= \chi_{(0, \lambda]}(q_{n-1} \CE_n(f) q_{n-1})
	\end{align*}
	satisfying
	\begin{enumerate}
		\item[(i)] $q_n \CE_n(f) q_n \leq \lambda q_n$ for all $n \in \N$,
		\item[(ii)] $q_n$ commutes with $q_{n-1} \CE_n(f) q_{n-1}$,
		\item[(iii)] suppose $q= \wedge_{n \in \Z} q_n$, then 
		\begin{align*}
			q f q \leq \lambda q, ~ \varphi(1-q) \leq \frac{\norm{f}_1}{\lambda} .
		\end{align*}
	\end{enumerate}
	
	Furthermore, write $p_n= q_{n-1} - q_n$ for $n \in \Z $ and observe that $p_n$'s are disjoint projections in $\CN$ and 
	\begin{equation}
		\sum_{n \in \Z} p_n = 1- q.
	\end{equation}
	
	\noindent Then, by the Calder\'on- Zygmund decomposition (cf. \cite{cadilhac2022spectral}) $f$ can be decomposed as
	\begin{equation}
		f= g +b_d + b_{off},
	\end{equation}
	where, denoting $f_j:= \CE_j(f)$ we have
	\begin{enumerate}
		\item $g= qfq + \sum_{j \geq 1} p_j f_j p_j$,
		\item $b_d= \sum_{j \geq 1} p_j(f- f_j) p_j$,
		\item $b_{off}= \sum_{j \geq 1} p_j(f- f_j) q_j + q_j(f- f_j) p_j$.
	\end{enumerate}
	
	\noindent The following proposition is proved in \cite{cadilhac2022spectral} and also crucial for our purpose.

	\begin{prop}
		The following are true.
		\begin{enumerate}
			\item $\norm{g}_1 \leq \norm{f}_1$ and $\norm{g}_\infty \lesssim \lambda$.
			\item $\sum_{j \geq 1} \norm{p_j(f- f_j) p_j}_1 \leq 2 \norm{f}_1$ and \\
			$\CE_j(p_j(f- f_j) p_j)= \CE_j( p_j(f- f_j) q_j + q_j(f- f_j) p_j) = 0$. 
		\end{enumerate}
	\end{prop}
	
	\noindent For $k \in \Z$, $Q \in \CQ_k$ and $x \in Q$, define $p_Q:= p_k(x)$ and 
	\begin{equation}\label{defn of zeta}
		\zeta:= \Big(\vee_{Q \in \CQ} p_Q \chi_{5Q}\Big)^\perp.
	\end{equation}
	
	We recall the following theorem.
	\begin{thm}\cite{galkazka2022sharp}\label{galkazka's estimate}
		Let $w$ be a weight satisfying the Muckenhoupt's $A_1$ condition and $f \in L^1_w(\CN)$. Then for any $\lambda>0$ we have $q \CE_n(f) q \leq \lambda$ for all $n$ and 
		\begin{align*}
			\lambda \varphi_w(1-q) \leq [w]_{A_1} \norm{f}_{1,w}.
		\end{align*}
	\end{thm}
	
	\noindent Now we prove the following important lemma.
	\begin{lem}\label{estimate of zeta perp}
		Let $w$ be an $A_1$-weight, $f \in L^1_w(\CN)_+$ and $\lambda>0$. Then we have 
		\begin{equation}
			\varphi_w (1 - \zeta) \lesssim [w]_{A_1}^2 \frac{\norm{f}_{1,w}}{\lambda}.
		\end{equation}
	\end{lem}
	
	\begin{proof}
		Observe that 
		\begin{align*}
			\varphi_w (1 - \zeta)
			&\leq \sum_{Q \in \CQ} \varphi_w (p_Q \chi_{5Q})\\
			& \leq \sum_{Q \in \CQ} [w]_{A_1} \frac{\abs{5Q}}{\abs{Q}} \varphi_w (p_Q \chi_{Q}) ~ (\text{by eq } \ref{weight prop})  \\
			& \lesssim [w]_{A_1} \sum_{j \geq 1} \varphi_w(p_j), \text{ since } \sum_{Q \in \CQ} p_Q \chi_{Q} = \sum_{j \geq 1} p_j \\
			& = [w]_{A_1} \varphi_w (1-q) \\
			& \lesssim [w]_{A_1}^2 \frac{\norm{f}_{1,w}}{\lambda} ~ (\text{by Theorem } \ref{galkazka's estimate}).
		\end{align*}
		This completes the proof of the lemma.
	\end{proof}

	\section{Weak type $(1,1)$ estimate}\label{weaktypeesti}
	
	This section is devoted to the weak type $(1,1)$ estimate of the non-commutative square function. 

	Recall that for the measure space $(\R^d, \CB, \mu)$ and $f \in L^1(\R^d, \CB, \mu)$, the square function is formulated as
	\begin{align*}
		Lf(x):= \Big( \sum_{k} |(M_k - \E_k) f(x)|^2 \Big)^{\frac{1}{2}}, ~ x \in \mathbb{R}^d,
	\end{align*}
	where, $M_k$ is defined as in eq. \ref{avg op} just by replacing $\CM$ by $\C$.
	In \cite{jones2003oscillation}, the authors proved that the operator $L$ satisfies strong type $(p,p)$ for $1<p \leq 2$ and weak type $(1,1)$.
	
	Let us we consider a Rademacher sequence $(\epsilon_k)$ on some probability measure space $(\Omega, P)$ and construct a sequence of operators
	\begin{align*}
		T_kf(x):= (M_k- \CE_k)f (x),~ f \in L^1_w(\CN, \varphi).
	\end{align*}
	Observe that by \cite{cadilhac2019noncommutative},  for a finite sequence $(T_kf)$ we have
	\begin{align*}
		\norm{(T_k f)}_{L_w^{1, \infty}(\CN; \ell_2^{rc})} \simeq \norm{\sum_k \epsilon_k T_k f}_{L_w^{1, \infty}(L^\infty(\Omega) \otimes \CN)}
	\end{align*}
	
	\noindent To estimate the weak type $(1,1)$ bounds for the operator
	\begin{align}\label{linearization}
		Tf(x):= \sum_k \epsilon_k T_kf (x), ~ f \in L^1_w(\CN, \varphi),
	\end{align}
	we first observe the following. \noindent Let $f \in L^1_w(\CN)_+$. Then we have that
	\begin{equation}
		\tilde{\varphi}_w(\abs{Tf}> \lambda) \leq \tilde{\varphi}_w(\abs{Tg}> \lambda/3) + \tilde{\varphi}_w(\abs{Tb_d}> \lambda/3) + 
		\tilde{\varphi}_w(\abs{Tb_{off}}> \lambda/3),
	\end{equation}
	where, $\tilde{\varphi}_w:= \int_\Omega \otimes \varphi_w$.
	So to prove the weak type $(1,1)$ inequality for the non-commutative square function $T$, it is enough to prove the following theorem.
	
	\begin{thm}\label{main thm}
		Let  $h \in \{g, b_d, b_{off}\}$ and $\lambda>0$. Then there exists a constant $C_1(w)>0$ such that 
		\begin{equation}
			\lambda \tilde{\varphi}_w ( \abs{Th}> \lambda) \lesssim C_1(w) \norm{f}_{1,w}.
		\end{equation}
	\end{thm}

	Let $n \in \Z$ and $B \subset \R^d$ be an Euclidean ball, and let $\partial B$ denote the boundary of the ball. Define the following set
	\begin{equation}
		\CI (B,n)= \cup_{Q \in \CQ_n, ~ \partial B \cap Q \neq \phi}  ~Q \cap B.
	\end{equation}
	Furthermore, for integer  $k<n$ and $x \in \R^d$ define,
	\begin{equation}
		M_{k,n} h(x):= \frac{1}{\abs{B_k}} \int_{\CI(B_k+x,n)} h(y) dy, ~ h \in L^p_w(\CN)_+.
	\end{equation}
	
	\noindent Now we prove the following lemma which is crucial for our case.
	
	\begin{lem}\label{main lemma}
		Let $1 \leq p \leq \infty$. Then for $h \in L^p_w(\CN)$ and $k < n$ we have 
		\begin{align*}
			\norm{M_{k,n} h}^p_{p, w} \lesssim C(w) [w]_{A_1} 2^{(k-n)\delta} \norm{h}^p_{p, w}.
		\end{align*}
		Moreover, if $h \in L^p_w(\CN)_+$, then 
		\begin{align*}
			\norm{M_{k,n} h}^p_{p, w} \lesssim C(w) [w]_{A_1} 2^{(k-n)\delta} \norm{\CE_n(h)}^p_{p, w}.
		\end{align*}
	\end{lem}
	
	\begin{proof}
		Note that 
		\begin{align*}
			\norm{M_{k,n} h}_{p, w}^p = \int_{\R^d}  \norm{M_{k,n} h(x)}_p^p w(x) dx.
		\end{align*}
		Furthermore,
		\begin{align*}
			\norm{M_{k,n} h(x)}_p 
			&\leq \frac{1}{\abs{B_k}} \int_{\CI(B_k+x,n)} \norm{h(y)}_p dy\\
			&\leq \frac{\abs{\CI(B_k+x,n)}^{1/ p'}}{\abs{B_k}} \Big( \int_{\CI(B_k+x,n)} \norm{h(y)}^p_p dy \Big)^{1/p}
		\end{align*}
		where, $\frac{1}{p} + \frac{1}{p'} =1$. Hence,
		\begin{align*}
			\norm{M_{k,n} h}_{p, w}^p
			\leq \frac{(2^{-n} 2^{(d-1)(-k)})^{p/p'}}{(2^{-kd})^p} \int_{\R^d} \Big(\int_{\CI(B_k+x,n)} \norm{h(y)}^p_p dy\Big) w(x) dx \\
			= \frac{(2^{-n} 2^{(d-1)(-k)})^{p/p'}}{(2^{-kd})^p} \int_{\R^d} \norm{h(y)}^p_p \Big(\int_{J_{n,y}}  w(x) dx \Big) dy,
		\end{align*}
		where, $J_{n,y}:= \{x \in \R^d : y \in \CI(B_k+x,n)\}$. Also note that $\abs{J_{n,y}} \leq 2^{-n} 2^{(d-1)(-k)}$. Let $Q_y$ be the cube centred at $y$ of side length $2^{-k+1}$ containing $J_{n,y}$. Therefore,
		\begin{align*}
			\norm{M_{k,n} h}_{p, w}^p
			&\leq \frac{(2^{-n} 2^{(d-1)(-k)})^{p/p'}}{(2^{-kd})^p} \int_{\R^d} \norm{h(y)}^p_p w(J_{n,y}) dy \\
			& \lesssim C(w) \frac{(2^{-n} 2^{(d-1)(-k)})^{p/p'}}{(2^{-kd})^p} \int_{\R^d} \norm{h(y)}^p_p w(Q_y) \left( \frac{\abs{J_{n,y}}}{\abs{Q_y}}\right)^\delta dy\\
			& \leq C(w) [w]_{A_1} \frac{(2^{-n} 2^{(d-1)(-k)})^{p/p'}}{(2^{-kd})^p} \int_{\R^d} \norm{h(y)}^p_p \frac{\abs{J_{n,y}}^\delta}{\abs{Q_y}^{\delta-1}} w(y) dy\\
			& \leq  C(w) [w]_{A_1} \frac{(2^{-n} 2^{(d-1)(-k)})^{p/p'}}{(2^{-kd})^p}  2^{-n} 2^{(d-1)(-k)} \int_{\R^d} \norm{h(y)}^p_p \left(\frac{\abs{J_{n,y}}}{\abs{Q_y}} \right)^{\delta-1} w(y) dy \\
			&= C(w) [w]_{A_1} 2^{(k-n)p} 2^{(k-n)(\delta-1)} 2^{d(1-\delta)} \norm{h}^p_{p,w}.\\
			&\lesssim C(w) [w]_{A_1} 2^{(k-n)(p+\delta-1)} \norm{h}^p_{p,w}\\
			&\lesssim C(w) [w]_{A_1} 2^{(k-n)\delta} \norm{h}^p_{p,w}.
		\end{align*}
		Moreover, if $h \in L^p_w(\CN)_+$, then observe that 
		\begin{align*}
			M_{k,n} h \leq \frac{1}{\abs{B_k}} \sum_{ Q \in \CQ_n, \partial (B_k +x) \cap Q \neq \phi } \int_{Q} \CE_n(h)(y) dy.
		\end{align*}
		Hence, a similar argument as above we conclude that 
		\begin{align*}
			\norm{M_{k,n} h}^p_{p, w} \lesssim C(w) [w]_{A_1} 2^{(k-n)\delta} \norm{\CE_n(h)}^p_{p, w}.
		\end{align*}
	\end{proof}
	
	\subsection*{Estimate for the bad part}
	
	\subsubsection*{\underline{Estimate for $Tb_d$}} Recall the definition of $\zeta$ from Eq. \ref{defn of zeta} and observe that
	\begin{equation}
		Tb_d = (1-\zeta) Tb_d (1-\zeta) + \zeta Tb_d (1-\zeta) + (1-\zeta) Tb_d \zeta + \zeta Tb_d \zeta.
	\end{equation}
	
	\noindent Therefore, we must have 
	\begin{align*}
		\tilde{\varphi}_w(\abs{Tb_d}> \lambda/3) 
		&\lesssim \varphi_w (1 - \zeta) + \tilde{\varphi}_w ( \abs{\zeta Tb_d \zeta}> \lambda/9)\\
		&\lesssim [w]_{A_1}^2\frac{\norm{f}_{1,w}}{\lambda} + \tilde{\varphi}_w ( \abs{\zeta Tb_d \zeta}> \lambda/9)~ (\text{by Lemma } \ref{estimate of zeta perp}).
	\end{align*}

	Let us first recall the almost orthogonality lemma from \cite{jones2003oscillation}, \cite{hong2017vector} which is crucial for our purpose.
	
	\begin{lem}\label{almost orthogonality}
		Let for each $k$, $S_k$ be a bounded operator on $L^2$ and $h \in L^2$. If $(u_n)_{n}$ and $(v_n)_n$ be two sequences in $L^2$ such that $h= \sum_n u_n$ and $\sum_n \norm{v_n}_2^2< \infty$, then 
		\begin{align*}
			\sum_k \norm{S_kh}_2^2 \leq \kappa^2 \sum_n \norm{v_n}_2^2
		\end{align*}
		provided that there is a sequence $(\kappa(j))_j$ of positive numbers with $\kappa= \sum_j \kappa(j) < \infty$ such that 
		\begin{align*}
			\norm{S_k(u_n)}_2 \leq \kappa(n-k) \norm{v_n}_2
		\end{align*}
		for all $n,k$.
	\end{lem}

	\begin{prop}
		The following is true.
		\begin{equation}
			\tilde{\varphi}_w ( \abs{\zeta Tb_d \zeta}> \lambda/9) \lesssim C(w) [w]_{A_1}^2 \frac{\norm{f}_{1,w}}{\lambda}.
		\end{equation}
	\end{prop}
	
	\begin{proof}
		Observe that by Chebychev's inequality it is enough to show that 
		\begin{equation}
			\norm{\zeta Tb_d \zeta}_{2,w}^2 \lesssim \lambda^2 \sum_{n} \norm{p_n}_{2,w}^2.
		\end{equation}
		Furthermore, note that 
		\begin{equation}
			\sum_{n} \norm{p_n}_{2,w}^2= \sum_{n} \norm{p_n}_{1,w}= \tilde{\varphi}_w(\sum_{n} p_n)= \tilde{\varphi}_w(1-q) \lesssim [w]_{A_1} \frac{\norm{f}_{1,w}}{\lambda}.
		\end{equation}
		Also,
		\begin{align*}
			\norm{\zeta Tb_d \zeta}_{2,w}^2 = \sum_{k} \norm{\zeta (M_k - \CE_k)(b_d) \zeta}_{2,w}^2.
		\end{align*}
		Therefore, it is enough to prove that 
		\begin{equation}
			\sum_{k} \norm{\zeta (M_k - \CE_k)(b_d) \zeta}_{2,w}^2 \lesssim \lambda^2 \sum_{n} \norm{p_n}_{2,w}^2.
		\end{equation}
		Observe that $b_d:= \sum_n b^d_n$, where, $b^d_n:= p_n(f-f_n)p_n$. Now we invoke Lemma \ref{almost orthogonality} and see that it is further enough to prove 
		\begin{align*}
			\norm{\zeta (M_k - \CE_k)(b^d_n) \zeta}_{2,w}^2 \lesssim 2^{-2 \abs{k-n} \delta} \lambda^2 \norm{p_n}_{2,w}^2.
		\end{align*}
		Now we divide the proof in two subcases, that is for $k\geq n$ and $k<n$. For $k\geq n$, following the same argument as in \cite[Proposition 3.9]{hong2021noncommutative} it can be shown that 
		\begin{equation}
			\zeta (M_k - \CE_k)(b^d_n) \zeta = 0.
		\end{equation}
		Now for $k<n$,  $\CE_k(b_n^d)= \CE_k \CE_n(b^d_n)=0$. Hence, in this case observe that $M_k b^d_n= M_{k,n} b^d_n$ . Indeed,
		\begin{align*}
			M_k b_n^d(x)
			&= \frac{1}{\abs{B_k}} \int_{B_k +x} b^d_n(y) dy\\
			&= \frac{1}{\abs{B_k}} \sum_{Q \in \CQ_n} \int_{B_k +x \cap Q} b^d_n(y) dy\\
			&= \frac{1}{\abs{B_k}} \int_{\CI(B_k+x, n)} b^d_n(y) dy= M_{k,n} b^d_n(x)
		\end{align*}
		Therefore, it suffices to prove that 
		\begin{equation}
			\norm{M_{k,n} b^d_n}_{2,w} \lesssim C(w) 2^{(k-n)\delta} [w]_{A_1} \lambda \norm{p_n}_{2,w},
		\end{equation}
		which now follows from Lemma \ref{main lemma} and Cuculescu's construction.
	\end{proof}
	
	\subsubsection*{\underline{Estimate for $Tb_{off}$}}
	
	Let us now recall a few definitions from \cite{cadilhac2022noncommutative} which will be helpful for the weak estimate of $Tb_{off}$. Recall that $B_n$ denote the open ball in $\R^d$ of radius equal to $2^{-n}$ for all $n \in \Z$.
	
	\begin{defn}
		For a bounded set $K \subseteq \R^d$ and $E \subseteq \R^d$, define the $K$-boundary of $E$ as
		\begin{align*}
			\partial_{K}(E):= \bigcup_{y \in \R^d}\{K +y: (K+y) \cap E \neq \emptyset \text{ and } (K+y) \cap E^c \neq \emptyset\},
		\end{align*}
		where $E^c:= \R^d \setminus E$.
	\end{defn}
	\noindent	Further recall that $\CQ_n$ denote the collection of all dyadic cubes of $\R^d$ of $n$-th generation. Now for $E \subseteq \R^d$, let us define the set
	\begin{align*}
		\partial_{\CQ_n}(E):= \bigcup \{Q \in \CQ_n : Q \cap E \neq \emptyset, Q \not\subseteq E\}.
	\end{align*}
	
	\begin{prop}\label{geom prop}
		Let $n,k \in \N$ and $x \in \R^d$. Then, the following properties hold.
		\begin{enumerate}
			\item[(i)] $\partial_{\CQ_n}(B_k+x)\subseteq \partial_{B_n} (B_k +x)$.
			\item[(ii)] $\partial_{B_n}(B_k + x)= \partial_{B_n}(B_k) + x$.
		\end{enumerate}
	\end{prop}
	
	\begin{proof}
		\emph{(i):} Let $Q \in \CQ_n$ such that $Q \cap (B_k+x) \neq \emptyset$ and $Q \not\subseteq B_k+x$. Since $Q$ is a cube of side length $2^{-n}$ and $B_n$ is a ball of radius $2^{-n}$, we conclude that there exists $y \in \R^d$ such that $Q \subseteq B_n +y$. Therefore, we get the result.\\
		
		\emph{(ii):} Let $z \in \partial_{B_n}(B_k+x)$. Then, there exists $y \in \R^d$ such that $z \in B_n +y$ with $(B_n +y ) \cap (B_k + x) \neq \emptyset $ and $(B_n +y ) \cap (B_k + x)^c \neq \emptyset $. Therefore, $z-x \in B_n + (y-x)$ and $(B_n + (y-x)) \cap B_k \neq \emptyset$ and $(B_n + (y-x)) \cap B_k^c \neq \emptyset$. Hence, we obtain $\partial_{B_n}(B_k + x) \subseteq \partial_{B_n}(B_k) + x$. The reverse direction follows similarly.
	\end{proof}
	
	We also recall the following observation from \cite{cadilhac2022noncommutative}, which we are going to use in the sequel.
	
	\begin{prop}\cite[Lemma 4.4]{cadilhac2022noncommutative}\label{Cadilhac prop}
		Let $k \in \N$ and $E$ be the union of all $\CQ_k$-atoms and $K \subset E$ and $\lambda>0$, then
		\begin{align*}
			\norm{\int_K p_k f q_k}_1 \leq \lambda 2 \varphi(\chi_E p_k f) \text{ and } \norm{\int_K q_k f p_k}_1 \leq \lambda 2 \varphi(\chi_E p_k f)
		\end{align*}
	\end{prop}
	
	\noindent We are now ready to prove the estimate for $Tb_{off}$. But before that we need the following lemma.

	\begin{lem}\label{main lem-off diag}
		Fix $n \geq 1$, then we have
		\begin{align*}
			\sum_{k: k< n} \norm{M_k b^{off}_n}_{1,w} \lesssim C(w)   [w]_{A_1}^2 \norm{f}_{1,w},
		\end{align*}
		where, $b_n^{off}:= p_n(f- f_n) q_n + q_n(f- f_n) p_n$ for all $n \geq 1$.
	\end{lem}
	
	\begin{proof}
		First observe that
		\begin{align*}
			M_k b_n^{off} (x)
			&= \frac{1}{\abs{B_k}} \int_{B_k +x } b^{off}_n (y) dy\\
			&= \frac{1}{\abs{B_k}} \int_{\partial_{\CQ_n}(B_k +x) \cap (B_k + x)} b^{off}_n(y) dy.
		\end{align*}
		Further note that $p_nf_nq_n=0$ for all $n \in \N$. Hence,
		\begin{align*}
			b_n^{off} = p_n f q_n + q_n f p_n.
		\end{align*}
		Therefore, by Proposition \ref{Cadilhac prop}, it is enough to consider $b_n^{off}= p_n f q_n$. Hence, we have
		\begin{align*}
			\norm{M_{k} b_n^{off}(x)}_1 
			&\leq \frac{2\lambda}{\abs{B_k}} \varphi (\chi_{\partial_{\CQ_n}(B_k +x)} p_n f )\\
			&\leq \frac{2\lambda}{\abs{B_k}} \int_{\R^d}  \chi_{(\partial_{B_n}(B_k)+x)}(y) \tau((p_n f )(y)) dy~ (\text{by Proposition. } \ref{geom prop}).
		\end{align*}
		Hence, for $k<n$ we have


		\begin{align*}
			\norm{M_k b_n^{off}}_{1,w}
			&= \int_{\R^d} \norm{M_{k} b_n^{off}(x)}_1 w(x) dx\\
			&\leq \frac{2\lambda}{\abs{B_k}} \int_{\R^d} \Big( \int_{\R^d}  \chi_{(\partial_{B_n}(B_k)+x)}(y) \tau((p_n f)(y)) dy \Big) w(x) dx\\
			&= \frac{2\lambda}{\abs{B_k}} \int_{\R^d} \Big( \int_{\R^d} \chi_{(y - \partial_{B_n}(B_k)}) (x) w(x) dx\Big) \tau((p_n f )(y)) dy\\
			&=  \frac{2\lambda}{\abs{B_k}} \int_{\R^d} w(y - \partial_{B_n}(B_k)) \tau((p_n f )(y)) dy\\
			&\lesssim C(w) \frac{2\lambda}{\abs{B_k}}  \int_{\R^d} w(Q_y) \frac{\abs{y - \partial_{B_n}(B_k)}^\delta}{\abs{Q_y}^\delta}  \tau((p_n f )(y)) dy\\
			&\lesssim C(w) \frac{2\lambda}{\abs{B_k}} \frac{\abs{\partial_{B_n}(B_k)}^\delta}{\abs{Q_y}^{\delta-1}} [w]_{A_1} \int_{\R^d}  \tau((p_n f )(y)) w(y) dy\\
			&\lesssim C(w) 2 \lambda 2^{(k-n)\delta} [w]_{A_1} \varphi_w(p_nf),
		\end{align*}
		where $Q_y$ is a cube centered at $y$ of side length $2^{-k+1}$ and containing $y - \partial_{B_n}(B_k)$. So, for $k<n$, we have

		\begin{equation*}
			\norm{M_k(b_n^{off})}_{1,w} \lesssim \lambda 2^{(k-n)\delta} C(w) [w]_{A_1} \varphi_w (p_n f).
		\end{equation*}
		
		Therefore,
		\begin{align*}
			\sum_{(k,n) \in \N \times \N, ~ k<n} \norm{M_k(b_n^{off})}_{1,w}
			&= \sum_{n=1}^\infty \sum_{k=1}^{n-1} \norm{M_k(b_n^{off})}_{1,w}\\
			&\lesssim \lambda C(w) [w]_{A_1} \sum_{n=1}^\infty \sum_{k=1}^{n-1} 2^{(k-n)\delta} \varphi_w (p_n f) \\
			&\lesssim \lambda C(w) [w]_{A_1} \sum_{n=1}^\infty \varphi_w (p_n f)\\
			& \lesssim \lambda C(w)  [w]_{A_1} \varphi_w((1-q)f) \\
			&\lesssim C(w) [w]_{A_1}^2 \norm{f}_{1,w}, ~\text{by Theorem} \quad\ref{galkazka's estimate}.
		\end{align*}
		This completes the proof.
	\end{proof}
	
	\begin{prop}\label{can prop}
		For all $Q \in \CQ$, the following cancellation property holds true.
		\begin{align*}
			x \in 5Q \Rightarrow p_Q \zeta(x)= \zeta(x) p_Q =0.
		\end{align*}
	\end{prop}
	
	\begin{prop}
		The following is true.
		\begin{equation}
			\tilde{\varphi}_w ( \abs{\zeta Tb_{off} \zeta}> \lambda/9) \lesssim C(w) [w]^2_{A_1} \frac{\norm{f}_{1,w}}{\lambda}.
		\end{equation}
	\end{prop}
	
	\begin{proof}
		By Chebychev's inequality,
		\begin{align*}
			\tilde{\varphi}_w ( \abs{\zeta Tb_{off} \zeta}> \lambda/9) \leq \frac{\norm{\zeta Tb_{off} \zeta}_{1,w}}{\lambda}.
		\end{align*}
		Furthermore,
		\begin{align*}
			\norm{\zeta Tb_{off} \zeta}_{1,w} 
			&\leq \sum_{n=1}^{\infty} \sum_k \norm{\zeta \epsilon_k (M_k- \CE_k) b_n^{off} \zeta}_{1,w}\\
			&\lesssim \sum_{n=1}^{\infty} \sum_k \norm{\zeta (M_k- \CE_k) b_n^{off} \zeta}_{1,w}.
		\end{align*}
		
		\noindent	For $k \geq n$ and $x \in \R^d$, we have 
		\begin{align*}
			\zeta(x) M_k b_n^{off}(x) \zeta(x) 
			&= \zeta(x) \frac{1}{\abs{B_k}} \int_{x+ B_k} b_n^{off}(y) \chi_{\{y \notin 5 Q_{x,n}\}} dy \zeta(x) \\
			&=0,
		\end{align*}
		since $x+ B_k \subseteq 5 Q_{x,n}$ and we have the cancellation property as in proposition \ref{can prop}. We further note that
		\begin{align*}
			\zeta(x) \CE_k b_n^{off}(x) \zeta(x)
			&= \zeta(x) \frac{1}{\abs{Q_{x,k}}} \int_{Q_{x,k}} b_n^{off}(y) \chi_{\{y \notin 5 Q_{x,n}\}} dy \zeta(x) \\
			&=0, \text{ since for } k\geq n,~ Q_{x,k} \subset Q_{x,n}.
		\end{align*}
		
		\noindent On the other hand, for $k<n$, note that $\CE_k(b_n^{off})= \CE_k \CE_n(b_n^{off})=0$. Hence we obtain
		\begin{align*}
			\sum_{n=1}^{\infty} \sum_k \norm{\zeta (M_k- \CE_k) b_n^{off} \zeta}_{1,w}
			\lesssim \sum_{k: k< n} \norm{M_k b^{off}_n}_{1,w}.
		\end{align*} 
		Therefore, the result follows from Lemma \ref{main lem-off diag}.
	\end{proof}

	\subsection*{Estimate for the good part}
	
	\begin{thm}
		The following is true.
		\begin{align*}
			\lambda \tilde{\varphi}_w(\abs{Tg}> \lambda) \leq \max \{[w]^2_{A_1}, [w]^3_{A_1}\} \norm{f}_{L^1_w(\CN)}
		\end{align*}
	\end{thm}
	
	\begin{proof}
		
		Fix an orthonormal basis $\Lambda$ of $L^2(\CM, \tau)$. Then observe that the set
		\begin{equation*}
			\{a \otimes \xi: a \in L^2_w(\R^d, \mu), \xi \in \Lambda \}
		\end{equation*}
		is total in $L^2_w(\CN)$. Now consider $h \in L^2_w(\CN) \cap \CN$, where $h= \sum_{1}^{l} a_i \otimes \xi_i$ with $a_i \in L^2(\R^d, w d\mu)$ and $\xi_i \in \Lambda$. Furthermore, notice that for all $k$
		\begin{equation*}
			(M_k- \CE_k)h = \sum_1^l \left( (M_k- \E_k)a_i \right) \otimes \xi_i.
		\end{equation*} 
		
		\noindent Therefore, using classical Khintchine inequalities we have
		
		\begin{align*}
			&\int_\Omega \norm{\sum_k \epsilon_k(s) (M_k - \CE_k) h}^2_{L^2_w(\CN)} dP(s)\\
			&= \int_\Omega \norm{\sum_1^l \left(\sum_k \epsilon_k(s) (M_k- \E_k)a_i \right) \otimes \xi_i}^2_{L^2_w(\CN)} dP(s)\\
			&= \sum_{i=1}^l \int_\Omega \left( \int_{\R^d}  \abs{ \sum_k \epsilon_k(s) \left(  (M_k- \E_k)  a_i \right)(x)  }^2 w(x) d \mu(x) \right) dP(s)\\
			&\approx \sum_{i=1}^l  \int_{\R^d} \sum_k \abs{ (M_k- \E_k) \left( a_i (x) \right) }^2  w(x) d \mu(x) \\
			&\lesssim [w]_{A_1}^2 \sum_{i=1}^l  \int_{\R^d} \abs{ a_i(x)}^2 w(x) d\mu(x) ~ \text{ (cf. \cite{krause2018weighted})}\\
			&= [w]_{A_1}^2 \norm{ \sum_1^l a_i \otimes \xi_i}^2_{L^2_w (\CN)}.
		\end{align*}
		
		Hence, we conclude that 
		
		\begin{align}\label{freel2bound}
			\norm{Th}_{L^2(L^\infty(\Omega) \otimes \CN_w)} \lesssim [w]_{A_1} \norm{h}_{L^2_w(\CN)}.
		\end{align}

		Hence, we can conclude that 
		\begin{align*}
			\tilde{\varphi}_w(\abs{Tg}> \lambda)
			&\lesssim \frac{[w]^2_{A_1} \norm{g}^2_{L^2_w(\CN)}}{\lambda^2}\\
			&\lesssim [w]^2_{A_1} \frac{\norm{g}_{L^1_w(\CN)}}{\lambda}, ~ \text{since } \norm{g}_\infty \lesssim \lambda.
		\end{align*} 
		
		Now, 
		\begin{align*}
			\norm{g}_{L^1_w(\CN)} 
			= \varphi (g w)
			&= \varphi_w (qfq) + \varphi_w \Big(\sum_{j \geq 1} p_j f_j p_j \Big)\\
			&= \varphi_w(qf) + \varphi_w \Big( \sum_j p_j \CE_j(f)\Big)\\
			&= \varphi_w(qf) + \varphi_w \Big( \sum_j  \CE_j( p_j f)\Big)\\
			&= \varphi_w(qf) + \varphi \Big( \sum_j  \CE_j( p_j f) w \Big)\\
			&= \varphi_w(qf) + \varphi \Big( \sum_j ( p_j f)  \CE_j(w) \Big)\\
			&= \varphi_w(qf) + \varphi \Big( \sum_j ( p_j f)  \frac{\CE_j(w)}{w} w \Big)\\
			&= \varphi_w(qf) + [w]_{A_1} \varphi_w \Big( \sum_j ( p_j f) \Big)\\
			&= \varphi_w(qf) + [w]_{A_1} \varphi_w ((1-q)f)\\
			&\leq \max\{1, [w]_{A_1}\} \norm{f}_{1,w}, ~\text{by Theorem}\ \ref{galkazka's estimate} \text{ and since }  \varphi_w(qf)\leq \norm{f}_{1,w}.
		\end{align*}
Therefore, the result follows.  
		\end{proof}
		\subsection*{Strong $(p,p)$-estimate of $T$ for $1<p<\infty$:}
		\begin{thm}\label{strongbdd}
		Let $1<p < \infty$ and $w$ be an $A_1$-weight. Then there exists a constant $C_p(w)$, depending only on $p$ and $w$ such that
		\begin{equation}
\left\Vert (T_k f) \right\Vert_{L^p(\mathcal N_w;\ell^{rc}_2)} \leq C_p(w) \left\Vert f \right\Vert_{p,w}~ \forall f \in L^p(\mathcal N_w),~ \text{when } 1<p< \infty.
\end{equation}   
		\end{thm}
		\begin{proof}
		Note that by Theorem \eqref{main thm} the operator $T$ defined as in \eqref{linearization} is weak type $(1,1)$	as an operator from $L^1(\mathcal{N}_w,\varphi_w)$ to $L^{1,\infty}(L^\infty(\Omega)\otimes\mathcal{N}_w,\tilde{\varphi}_w).$ Moreover, it also follows from \cite{krause2018weighted} that $T$ is a bounded operator from $L^2(\mathcal{N}_w,\varphi_w)$ to $L^2(L^\infty(\Omega)\otimes\mathcal{N}_w,\tilde{\varphi}_w).$ We refer to the computation preceding inequality \eqref{freel2bound} for more details on this strong $(2,2)$ bound for $T.$ Therefore, by non-commutative Marcinkiewicz interpolation theorem, we obtain that $T$ is a bounded map from $L^p(\mathcal{N}_w,\varphi_w)$ to $L^p(L^\infty(\Omega)\otimes\mathcal{N}_w,\tilde{\varphi}_w)$ for all $1<p<2$. Hence by duality we obtain that $T^*$ is also a bounded linear operator from $L^{p^\prime}(\mathcal{N}_w,\varphi_w)$ to $L^{p^\prime}(L^\infty(\Omega)\otimes\mathcal{N}_w,\tilde{\varphi}_w)$ for all $1<p<2$, where $\frac{1}{p}+ \frac{1}{p'}=1$. By an easy calculation one can see that $T^*=T$ and hence $T$ extends to a bounded linear map from $L^p(\mathcal{N}_w,\varphi_w)$ to $L^p(L^\infty(\Omega)\otimes\mathcal{N}_w,\tilde{\varphi}_w)$ for all $1<p<\infty.$ The proof is now completed by non-commutative Khintchine's inequality \cite{Pisierricard2017} and \cite{PISIER2009Khintchine}.
		\end{proof}

	\textbf{Acknowledgement:} The first author acknowledges the DST-INSPIRE Faculty Fellowship DST/INSPIRE/04/2020/001132 and Prime Minister Early Career Research Grant Scheme ANRF/ECRG/2024/000699/PMS. The second author is thankful for the RA-I honorarium from ISI Delhi. We are also thankful to an anonymous referee
for their helpful comments and suggestions significantly improving the paper.

%

\providecommand{\bysame}{\leavevmode\hbox to3em{\hrulefill}\thinspace}
\providecommand{\MR}{\relax\ifhmode\unskip\space\fi MR }
\providecommand{\MRhref}[2]{%
	\href{http://www.ams.org/mathscinet-getitem?mr=#1}{#2}
}
\providecommand{\href}[2]{#2}

\end{document}